\documentclass[10pt]{article}
\usepackage[dvips]{graphicx}
\usepackage{amssymb,color}
\usepackage[latin1]{inputenc}
\usepackage{epic}
\usepackage{pstricks}
\usepackage{pst-node}
\usepackage{pstricks-add}
\usepackage[latin1]{inputenc}
\usepackage{color}
\usepackage{lettrine}
\usepackage{calc,pifont}
\usepackage[noindentafter,calcwidth]{titlesec}
\usepackage{pstricks}
\usepackage{pst-plot}
\usepackage{pst-node}
\usepackage{amsmath}
\usepackage{amssymb}
\usepackage{amsthm}
\usepackage{subfigure}
\usepackage{lineno}

\def\BBox{\kern  -0.2cm\hbox{\vrule width 0.2cm height 0.2cm}}

\newtheorem{example}{Example}
\newtheorem{remark}{Remark}

\newtheorem{teo}{Theorem}[section]
\newtheorem{coro}[teo]{Corollary}
\newtheorem{lema}[teo]{Lemma}

\newtheorem{conjecture}[teo]{Conjecture}
\theoremstyle{definition}

\theoremstyle{remark}

\title{On strong Skolem starters}
\author{Adrián Vázquez-Ávila\thanks{adrian.vazquez@unaq.edu.mx}\\
{\small Subdirección de Ingeniería y Posgrado}\\
{\small Universidad Aeronáutica en Querétaro}\\
}

\date{}

\begin{document}
%\linenumbers
\maketitle
%%%%%%%%%%%%%%%%%%%%%%%%%%%%%%%%%%%%%%%%%%%%%%
\begin{abstract}
In this note we present an alternative (simple) construction of cardioidal starters (strong and Skolem) for $\mathbb{Z}_{q^n}$, where $q\equiv3$ (mod 8) is an odd prime number and $n\geq1$ is an integer number; also for $\mathbb{Z}_{pq}$ and $\mathbb{Z}_{p^n}$, for infinitely many odd primes $p,q\equiv1$ (mod 8) and $n\geq1$ an integer number. This cardioidal starters can be obtained by a more general result given by Ogandzhanyants et al. [O. Ogandzhanyants, M. Kondratieva and N. Shalaby, \emph{Strong Skolem starters}, J. Combin. Des. {\bf 27} (2018), no. 1, 5--21]. 
\end{abstract}

%%%%%%%%%%%%%%%%%%%%%%%%%%%%%%%%%%%%%%%%%%%%%%

%%%%%%%%%%%%%%%%%%%%%%%%%%%%%%%%%%%%%%%%%%%%%%%%%%%%%%
%%%%%%%%%%%%%%%%%%%%%%%%%%%%%%%%%%%%%%%%%%%%%%%%%%%%%%INTRODUCTION
\section{Introduction}

Let $G$ be a finite additive abelian group of odd order $n=2k+1$, and let $G^*=G\setminus\{0\}$ be the set of non-zero elements of $G$. A \emph{starter} for $G$ is a set $S=\{\{x_i,y_i\}\}_{i=1}^k$ such that $\left\{x_1,\ldots,x_k,y_1,\ldots,y_k\right\}=G^*$
and $\left\{\pm(x_i-y_i)\right\}_{i=1}^k=G^*$. Moreover, if all elements of $\left\{x_i+y_i:i=1,\ldots,k\right\}\subseteq G^*$ are different and non-zero, then $S$ is called \emph{strong starter} for $G$. To see some works related to strong starters, the reader may consult \cite{Avila,MR0325419,dinitz1984,MR1010576,MR0392622,MR1044227,MR808085,MR0249314,MR0260604}. 

\begin{example}
The set $$S=\{\{17,18\},\{2,4\},\{3,6\},\{11,15\},\{9,14\}, \{7,13\},\{5,12\},\{8,16\},\{1,10\}\}$$ is a strong Skolem starter for $\mathbb{Z}_{19}$.	
\end{example}

Strong starters were first introduced by Mullin and Stanton in \cite{MR0234587} in constructing of Room
squares. Later were useful to construct other combinatorial designs: Room cubes \cite{MR633117}, Howell designs \cite{MR728501}, Kirkman triple systems \cite{MR808085,MR0314644}, Kirkman squares and cubes \cite{MR833796,MR793636}.  Also, there are interesting applications to the problem of one-factorization of complete graphs and problems related to it, see for example \cite{Anderson2,Anderson3,Anderson4,Bao,MR1010576,Ihrig,Kobayashi,Seah,Seah2,AvilaC4free,AvilaC4Nofree}.

Let $n=2k+1$, and $1<2<\cdots<2k$ be the order of the elements in $\mathbb{Z}_n^*$. A starter $S$ for $\mathbb{Z}_n$ is \emph{Skolem}, if $S$ can be written as $S=\{\{x_i,y_i\}\}_{i=1}^k$ such that $y_i>x_i$ and $y_i-x_i=i$ (mod n), for $i=1,\ldots,k$. In \cite{ShalabyThesis}, was proved that, the Skolem starter for $\mathbb{Z}_n$ exits, if and only if, $n\equiv1$ or 3 (mod 8). A starter which is both Skolem and strong is called \emph{strong Skolem starter}.

In 1991, Shalaby conjectured in \cite{ShalabyThesis} the following:

\begin{conjecture}\cite{ShalabyThesis}
If $n\equiv1,3$ (mod 8) and $n \geq11$, then $\mathbb{Z}_n$ admits a strong Skolem starter.	
\end{conjecture}

In that same work, Shalaby provided examples of strong Skolem starters for $\mathbb{Z}_n$ for admissible orders between 11 and 57.

Ogandzhanyants et al. in \cite{Shalaby} defined the following: A starter $S$ for $\mathbb{Z}_n$ is called \emph{cardioidal} if all elements of $S$ are of the type $\{i,2i \mbox{ (mod $n$)}\}$. In that same paper, \cite{Shalaby}, proved the following:
\begin{lema}\cite{Shalaby}\label{lemma:Skolem}
Let $n\equiv1,3$ ( mod 8). If $S$ is a cardioidal starter for $\mathbb{Z}_n$, then $S$ is Skolem. Moreover, $S$ is strong if and only if $n\not\equiv0$ $(\mbox{ mod } 3)$. 
\end{lema}

Ogandzhanyants et al. in \cite{Shalaby} proved that, if $n=\Pi_{i=1}^{k}p_i^{\alpha_i}$, where $p_i>3$ is an odd prime number such that $ord(2)_{p_i}\equiv 2$ (mod 4) and $\alpha_i$ is a non-negative integer, for all $i=1,\ldots,k$, then $\mathbb{Z}_n$ admits a strong Skolem starter (a cardioidal starter). The author in \cite{AvilaSkolem} presented an alternative (simple) construction of cardioidal starters (strong and Skolem) for $\mathbb{Z}_p$, where $p\equiv3$ (mod 8) is an odd prime number. On the other hand, the author in \cite{AvilaAKCE} gave an alternative (simple) construction of cardioidal starters (strong and Skolem) for $\mathbb{Z}_{pq}$, where $p,q\equiv3$ (mod 8) are odd prime numbers. In this paper, we presented an alternative (simple) construction of cardioidal starters (strong and Skolem) for $\mathbb{Z}_{p^n}$, where $p\equiv3$ (mod 8) with $n\geq1$ an integer. This result generalizes the work given by the author in \cite{AvilaSkolem}. Also, we presented an alternative (simple) construction of cardioidal starters (strong and Skolem) for $\mathbb{Z}_{p}$, $\mathbb{Z}_{p^n}$ and $\mathbb{Z}_{pq}$, for infinitely many odd primes $p,q\equiv1$ (mod 8) and $n\geq1$ an integer number. This cardioidal starters can be obtained by a more general result given by Ogandzhanyants [O. Ogandzhanyants, M. Kondratieva and N. Shalaby, \emph{Strong Skolem starters}, J. Combin. Des. {\bf 27} (2018), no. 1, 5--21]. 
%%%%%%%%%%%%%%%%%%%%%%%%
%%%%%%%%%%%%%%%%%%%%%%%%
%%%%%%%%%%%%%%%%%%%%%%%%

\section{Strong Skolem starters for $\mathbb{Z}_p$}\label{sec:quadratic}

Let $p$ be an odd prime number. An element $x\in\mathbb{Z}_p^*$ is called a \emph{quadratic residue}, if there exists an element $y\in\mathbb{Z}_p^{*}$ such that $y^2=x$. If there is no such $y$, then $x$ is called a \emph{non-quadratic residue.} The set of quadratic residues of $\mathbb{Z}_p^{*}$ is denoted by $QR(p)$ and the set of non-quadratic residues is denoted by $NQR(p)$. It is well-known that $QR(p)$ is a cyclic subgroup of $\mathbb{Z}_p^{*}$ of cardinality $\frac{p-1}{2}$ (see for example \cite{MR2445243}). As well as it is well-known that, if either $x,y\in QR(q)$ or $x,y\in NQR(q)$, then $xy\in QR(q)$, and if $x\in QR(q)$ and $y\in NQR(q)$, then $xy\in NQR(q)$. For more details of this kind of results the reader may consult \cite{burton2007elementary,MR2445243}.

Horton in \cite{MR623318} proved the following (see also \cite{Avila}):

\begin{lema}\cite{MR623318} Let $p\equiv3$ (mod 4) be an odd prime number, with $p\neq3$. If $\beta\in NQR(p)\setminus\{-1\}$, then the following set
\begin{eqnarray*}\label{strong_1}
S_\beta=\left\{\{x,\beta x\}: x\in QR(p)\right\},
\end{eqnarray*} 
is a strong starter for $\mathbb{Z}_p$. 
\end{lema}

The author in \cite{AvilaSkolem} proved the following:

\begin{teo}\cite{AvilaSkolem}\label{thm:SkolemAdrian} 
Let $p\equiv3$ (mod 8) be an odd prime number, with $p\neq3$. If $\beta\in\{2,2^{-1}\}$, then the strong starter
\begin{eqnarray*}\label{strong_1}
S_\beta=\left\{\{x,\beta x\}: x\in QR(p)\right\},
\end{eqnarray*} 
is a strong Skolem starter for $\mathbb{Z}_p$. 
\end{teo}

Notice that, the starters given by Theorem \ref{thm:SkolemAdrian} are cardioidal starters, then by Lemma \ref{lemma:Skolem} are strong and Skolem. 

\begin{remark}
It is well-known that, if $S=\{\{x_i,y_i\}:i=1,\ldots,k\}$ is a starter for a finite additive Abelian group $G$ of odd order $n=2k+1$, then $-S=\{\{-x_i,-y_i\}:i=1,\ldots,k\}$ is also a starter for $G$. Hence, the starter $-S_2=\{\{-x,-x\beta:x\in QR(p)\}\}=\{\{y,2y\}:y\in NQR(p)\}$ (given that $-1\in NQR(p)$, then $\beta\to\beta^2$ is a bijection between $NQR(p)$ and $QR(p)$) is an cardioidal starter for $\mathbb{Z}_p$ (strong and Skolem by Lemma \ref{lemma:Skolem}). Since $2^{-1}\in NQR(p)$, then $S_{2^{-1}}=-S_2$. 
\end{remark}

Hence, Theorem \ref{thm:SkolemAdrian} gives an alternative (simple) construction of cardioidal starter for $\mathbb{Z}_p$, where $p\equiv3$ (mod 8) be an odd prime number with $p\neq3$. This cardioidal starters can be obtained from \cite{Shalaby}.

Let $p=2^kt+1$ be an odd prime, where $k>1$ a positive integer and $t$ an odd integer greater than 1. Let $r$ be a generator of $\mathbb{Z}_p^*$. Let's denote by $C_0^p=\langle r^\Delta\rangle_p$, where $\Delta=2^k$, the multiplicative subgroup of $\mathbb{Z}_p^*$ generated by $r^\Delta$ of order $t$. If
$C_j^p=r^jC_0^p$, for $j=0,\ldots,\Delta-1$, then $C_0^p,\ldots,C_{\Delta-1}^p$ is a partition of $\mathbb{Z}_p^*$. In this case, $C_i^p$ are known as \emph{cyclotomic classes of order $\Delta$} of $\mathbb{Z}_p^*$.

According with \cite{Bao}, the following set $T_\beta=\left\{\{x,\beta x\}: x\in B\right\}$, where $B=\bigcup_{j=0}^{2^{k-1}-1}C_j^p$, is a strong starter for $\mathbb{Z}_p$, for every $\beta\in C_{2^{k-1}}^p\setminus\{-1\}$.

Hence, we have the following:

\begin{teo}\label{teo:Smkolem_q} 
Let $p=2^kt+1$ be an odd prime with $k\geq3$ a positive integer, $t$ an odd integer greater than 1, and $B=\bigcup_{j=0}^{2^{k-1}-1}C_j^p$ as above. If $2\in C_{2^{k-1}}^p$, then the following sets $T_2=\left\{\{x,2 x\}: x\in B\right\}$ and $T_{2^{-1}}=\left\{\{x,2^{-1} x\}: x\in B\right\}$ are a strong Skolem starter for $\mathbb{Z}_p$.
\end{teo}
\begin{proof}
Notice that $T_2$ and $T_{2^{-1}}$ are cardioidal starters (since $2\in C_{2^{k-1}}^p$, then $2^{-1}\in C_{2^{k-1}}^p$). Hence, by Lemma \ref{lemma:Skolem} the cardioidal starters are strong and Skolem.
\end{proof}

Notice that there exist infinitely many primes $p=2^kt+1$, with $k\geq3$ a positive integer and $t$ be an odd integer greater than 1, such that $2\in C_{2^{k-1}}^p$. For example, there are infinitely many primes $p=8n+1$ such that $2\in C_4^p$, see for example \cite{Brauer}.

It is not hard to see that $T_{2^{-1}}=-T_2$ (given that $-1\in C_{2^{k-1}}^p$). 
Theorem \ref{teo:Smkolem_q} gives an alternative (simple) construction of cardioidal starter for $\mathbb{Z}_p$, for infinitely many odd primes $p\equiv1$ (mod 8). This cardioidal starters can be obtained from \cite{Shalaby}.
%%%%%%%%%%%%%%%%%%%%%%%%
%%%%%%%%%%%%%%%%%%%%%%%%
%%%%%%%%%%%%%%%%%%%%%%%%
\section{Strong Skolem starter for $\mathbb{Z}_{p^n}$}\label{sec:Multiplicative}

In this section, we use the notation and terminology given in \cite{Shalaby}. Let $G_n$ be the group of units of the ring $\mathbb{Z}_n$ (elements invertible with respect to multiplication). Whenever the group operation is irrelevant, it is consider $G_n$ and its cyclic multiplicative subgroups $\langle x\rangle_n$ in the set-theoretical sense and denote them by $\underline{G}_n$ and $\langle\underline{x}\rangle_n$, respectively. It is well-known that, $G_n$ is a cyclic group, if and only if, $n$ is 1, 2, 4, $p$, $p^k$ or $2p^k$, where $p$ is an odd prime number and $k$ is an integer greater than 1. Finally, it is well-known that $|G_{p^k}|=p^k-p^{k-1}$. For more details of this kind of results the reader may consult \cite{Cohen:1995}.

\begin{remark}\label{remark} 
Let $p\equiv3$ (mod 8) be an odd prime number and $n\geq1$ be an integer number. If $p^i\underline{G}_{p^{n-i}}=\{x\in\mathbb{Z}_{p^n}^*: gcd(x,p^n)=p^i\}$, for every $i=0,\ldots,n-1$, then the collection $\left\{p^i\underline{G}_{p^{n-i}}\right\}_{i=0}^{n-1}$ forms a partition of $\mathbb{Z}_{p^n}^*$ into $n$ subsets, since every element $x\in\mathbb{Z}_{p^n}^*$ lies in one and only one of these set see (proof of Theorem 4.9 of \cite{Shalaby}).
\end{remark}

Let $p$ be an odd prime power. It is well-known if $r\in\mathbb{F}_p^*$ is a primitive root of $\mathbb{F}_p^*$, then $r$ is a primitive root of $G_{p^n}$, all integers $n$ greater than 1, unless $r^{p-1}\equiv1$ (mod $p^2$). In that case, $r+p$ is a primitive root of $G_{p^n}$, see for example \cite{Cohen:1995}.

\begin{teo}\label{thm:main} 
Let $p\equiv3$ (mod 8) be an odd prime, $n\geq1$ be an integer number, and $r\in\mathbb{Z}_p$ be a primitive root of $\mathbb{Z}_{p}$ such that $r^{p-1}\not\equiv1$ (mod $p^2$). If $p^iS_{p^{n-i}}=\left\{\{p^ix,2p^ix\}:x\in\langle r^2\rangle_{p^{n-i}}\right\}$, for $i=0,\ldots,n-1$, then $S=\bigcup_{i=0}^{n-1}p^iS_{p^{n-i}}$ is a strong Skolem starter for $\mathbb{Z}_{p^n}$.
\end{teo}
\begin{proof}
Let $r\in\mathbb{Z}_p^*$ be a primitive root such that $r^{p-1}\not\equiv1$ (mod $p^2$); otherwise $r+p$ is a primitive root of $G_{p^n}$. Let $\alpha=r^2$. Hence, $\langle r^2\rangle_{p^i}$ is a cyclic subgroup of $G_{p^i}$ of order $\frac{|G_{p^i}|}{2}=\frac{1}{2}(p^i-p^{i-1})$, for all $i=1,\ldots,n$. Since $2\not\in\langle\alpha\rangle_{p}$, then $2\not\in\langle\alpha\rangle_{p^i}$, also, since $-1\not\in\langle\alpha\rangle_{p}$, then $-1\not\in\langle\alpha\rangle_{p^i}$, for all $i=1,\ldots,n$. By Remark \ref{remark}, the collection $\left\{p^i\underline{G}_{p^{n-i}}\right\}_{i=0}^{n-1}$ forms a partition of $\mathbb{Z}_{p^n}^*$, where $\underline{G}_{p^i}$ can be written as $\underline{G}_{p^i}=\langle \underline{x}\rangle_{p^i}\cup2\langle\underline{x}\rangle_{p^i}$. Let $p^iS_{p^{n-i}}=\left\{\{p^ix,2p^ix\}:x\in\langle\alpha\rangle_{p^{n-i}}\right\}$, for $i=0,\ldots,n-1$. It is easy to see that $\{\pm p^ix:x\in\langle\alpha\rangle_{p^{n-i}}\}=p^i\underline{G}_{p^{n-i}}$, for all $i=0,\dots,n-1$, since $2\langle\underline{x}\rangle_{p^i}=-\langle\underline{x}\rangle_{p^i}$. Hence, the set $S=\bigcup_{i=0}^{n-1}p^iS_{p^{n-i}}$ is a starter. Moreover, each element of $S$ is of type $\{x,2x \mbox{ (mod $p^j$)}\}$, for all $j=1,\ldots,n$. Then, by Lemma \ref{lemma:Skolem} the starter $S$ is strong and Skolem. 
\end{proof}

\begin{coro}\label{coro:final} 
Let $p\equiv3$ (mod 8) be an odd prime, $n\geq1$ be an integer number, and $r\in\mathbb{Z}_p$ be a primitive root of $\mathbb{Z}_{p}$ such that $r^{p-1}\not\equiv1$ (mod $p^2$). If $p^i\hat{S}_{p^{n-i}}=\left\{\{p^ix,2^{-1}p^ix\}:x\in\langle\alpha\rangle_{p^{n-i}}\right\}$, for $i=0,\ldots,n-1$, then $\hat{S}=\bigcup_{i=0}^{n-1}p^i\hat{S}_{p^{n-i}}$ is a strong Skolem starter for $\mathbb{Z}_{p^n}$.
\end{coro}

It is not hard to see that $\hat{S}=-S$, since $2\langle\underline{x}\rangle_{p^i}=-\langle\underline{x}\rangle_{p^i}$, for all $i=0,\dots,n-1$. Theorem \ref{thm:main} gives an alternative (simple) construction of cardioidal starter for $\mathbb{Z}_{p^n}$, where  $p=2^kt+1$ is an odd prime with $k\geq3$ a positive integer, $t$ an odd integer greater than 1 and $n$ is an integer greater than 1.

\begin{teo}\label{thm:New_skolem_p^n}
	Let $p=2^kt+1$ be an odd prime with $k\geq3$ a positive integer, with $t$ an odd integer greater than 1. If $2\in C_{2^{k-1}}^p$ and $n$ is an integer greater than 1, then the set $T=\bigcup_{i=0}^{n-1}p^iT_{p^{n-i}}$, where $$p^iT_{p^{n-i}}=\left\{\{p^ix,2p^ix\}:x\in r^{j}\langle r^{2^k}\rangle_{p^i}, j=0,\ldots,2^{k-1}-1\right\},$$is a strong Skolem starter for $\mathbb{Z}_{p^n}$.
\end{teo}
\begin{proof}
Let $r\in\mathbb{Z}_p^*$ be a primitive root of $\mathbb{Z}_{p^n}$ such that $r^{p-1}\not\equiv1$ (mod $p^2$). 
Hence, $\langle r^{2^k}\rangle_{p^i}$ is a cyclic subgroup of $G_{p^i}$ of order $\frac{|G_{p^i}|}{2^k}$, for all $i=1,\ldots,n$. It is easy to see that, if $2\in C_{2^{k-1}}^p$ then $2\in r^{2^{k-1}}\langle r^{2^k}\rangle_{p^i}$ and if $-1\in C_{2^{k-1}}^p$ then $-1\in r^{2^{k-1}}\langle r^{2^k}\rangle_{p^i}$, for all $i=2,\ldots,n$. Since $\left\{p^i\underline{G}_{p^{n-i}}\right\}_{i=0}^{n-1}$ forms a partition of $\mathbb{Z}_{p^n}^*$, let  $$p^iT_{p^{n-i}}=\left\{\{p^ix,2p^ix\}:x\in r^{j}\langle r^{2^k}\rangle_{p^i}, j=0,\ldots,2^{k-1}-1\right\},$$for $i=0,\ldots,n-1$. It is easy to see that 
$$\{\pm p^ix:x\in r^{j}\langle r^{2^k}\rangle_{p^i},j=0,\ldots,2^{k-1}-1\}=p^i\underline{G}_{p^{n-i}},$$for all $i=0,\dots,n-1$. Hence, the set $T=\bigcup_{i=0}^{n-1}p^iT_{p^{n-i}}$ is a starter. Given that each element of $S$ is of type $\{x,2x \mbox{ (mod $p^j$)}\}$, for all $j=1,\ldots,n$, then, by Lemma \ref{lemma:Skolem} the starter $S$ is strong and Skolem. 
\end{proof}

\begin{coro}
Let $p=2^kt+1$ be an odd prime with $k\geq3$ a positive integer, with $t$ an odd integer greater than 1. If $2\in C_{2^{k-1}}^p$ and $n$ is an integer greater than 1, then the set $\hat{T}=\bigcup_{i=0}^{n-1}p^i\hat{T}_{p^{n-i}}$, where $$p^i\hat{T}_{p^{n-i}}=\left\{\{p^ix,2^{-1}p^ix\}:x\in r^{j}\langle r^{2^k}\rangle_{p^i}, j=0,\ldots,2^{k-1}-1\right\},$$ is a strong Skolem starter for $\mathbb{Z}_{p^n}$.	
\end{coro}

It is not hard to see that $\hat{T}=-T$ (given that $-1\in C_{2^{k-1}}^p$). 
Theorem \ref{thm:New_skolem_p^n} gives an alternative (simple) construction of cardioidal starter for $\mathbb{Z}_{p^n}$, where  $p=2^kt+1$ is an odd prime with $k\geq3$ a positive integer, $t$ an odd integer greater than 1 and $n$ is an integer greater than 1.
%%%%%%%%%%%%%%%%%%%%%%%%%%%%%%%%%%%%%%%%%%%%%%%
%%%%%%%%%%%%%%%%%%%%%%%%%%%%%%%%%%%%%%%%%%%%%%%
%%%%%%%%%%%%%%%%%%%%%%%%%%%%%%%%%%%%%%%%%%%%%%%%
\section{Strong Skolem starters for $\mathbb{Z}_{pq}$}\label{sec:Zpq}
Let $p$ and $q$ be odd prime numbers such that $p<q$ and $(p-1)\not|(q-1)$. We have $\underline{G}_{pq}=\{x\in\mathbb{Z}_{pq}^*: gcd(x,pq)=1\}$, with $|\underline{G}_{pq}|=(p-1)(q-1)$, see for example \cite{Cohen:1995}. Hence, $p\mathbb{Z}_q^*$, $q\mathbb{Z}_p^*$ and $\underline{G}_{pq}$ forms a partition of $\mathbb{Z}_{pq}^*$, since every element $x\in\mathbb{Z}_{pq}^*$ lies in one and only one of these sets. It's not hard to see that, if $r\in\mathbb{Z}_p^*$ is a primitive root, then $|\langle r \rangle_{pq}|=lcm(p-1,q-1)=\frac{(p-1)(q-1)}{2}$, since $(p-1)\not|(q-1)$.

\begin{lema}\label{lema:New_-1}
Let $t_1$ and $t_2$ be two odd integers greater than 1 such that $t_1<t_2$, and let $k\geq3$ be a integer. If $p=2^kt_1+1$ and $q=2^kt_2+1$ are prime numbers and $r\in\mathbb{Z}_p^*$ is such that $r\in NQR(p)$ and $r\in NQR(q)$, then $r^{\frac{(p-1)(q-1)}{2^{k+1}}}\equiv-1$ (mod $pq$). Moreover, if $r$ is a primitive root of $\mathbb{Z}_p^*$ and $\mathbb{Z}_q^*$, then $-1\in r^{2^{k-1}}\langle r^{2^k}\rangle_{pq}$.
\end{lema}
\begin{proof}
Recalling that $G_{m}$ is the group of units of $\mathbb{Z}_{m}$. It is well known that the map $\Psi:G_{pq}\to G_{p}\times G_{q}$, defined by $\Psi(k_{pq})=(k_p,k_q)$, is an isomorphism between $G_{pq}$ and $G_p\times G_q$, see for example \cite{Childs}. Since $p$ and $q$ are prime numbers, then $G_p=\mathbb{Z}_p^*$ and $G_q=\mathbb{Z}_q^*$. Let $r\in\mathbb{Z}_p^*$ such that $r\in NQR(p)$ and $r\in NQR(q)$. Hence, we have $r^{\frac{p-1}{2}}=-1$ (mod $p$) and $r^{\frac{q-1}{2}}=-1$ (mod $q$). Then
	\begin{eqnarray*}
		\Psi(r^{\frac{(p-1)(q-1)}{2^{k+1}}}_{pq})&=&(r^{\frac{(p-1)(q-1)}{2^{k+1}}}_p,r^{\frac{(p-1)(q-1)}{2^{k+1}}}_q)=(r_p^{\frac{p-1}{2}\cdot\frac{q-1}{2^k}},r_q^{\frac{q-1}{2}\cdot\frac{p-1}{2^k}}),\\
		&=&(-1_p,-1_q)=\Psi(-1_{pq})
	\end{eqnarray*}
	since $\frac{p-1}{2^k}$ and $\frac{q-1}{2^k}$ are odd integers. Hence, if $r$ is a primitive root of $\mathbb{Z}_p^*$ and $\mathbb{Z}_q^*$, then $-1_p\in r^{2^{k-1}}\langle r^{2^k}\rangle_p$ and $-1_q\in r^{2^{k-1}}\langle r^k\rangle_q$, which implies that $r^{2^{k-1}}\langle r^k\rangle_{pq}$.
\end{proof}

\begin{lema}\label{lema:New_2}
Let $t_1$ and $t_2$ be two odd integers greater than 1 such that $t_1<t_2$, and let $k\geq3$ be a integer. If $p=2^kt_1+1$ and $q=2^kt_2+1$ are prime numbers and $r\in\mathbb{Z}_p^*$ is a primitive root of $\mathbb{Z}_p^*$ and $\mathbb{Z}_q^*$ such that $2\in r^{2^{k-1}}\langle r^{2^k}\rangle_p$ and $2\in r^{2^{k-1}}\langle r^{2^k}\rangle_q$, then $2\in r^{2^{k-1}}\langle r^{2^k}\rangle_{pq}$.
\end{lema}
\begin{proof}
Let $\Psi:G_{pq}\to G_{p}\times G_q$ given by $\Psi(k_{pq})=(k_p,k_q)$, the isomorphism between $G_{pq}$ and $G_p\times G_q$, and let $r\in\mathbb{Z}_p^*$ be a primitive root of $\mathbb{Z}_p^*$ and $\mathbb{Z}_q^*$ such that $2\in r^{2^{k-1}}\langle r^{2^k}\rangle_p$ and $2\in r^{2^{k-1}}\langle r^{2^k}\rangle_q$. Assume that $2\not\in r^{2^{k-1}}\langle r^{2^k}\rangle_{pq}$. 
Hence, $2\in r^j\langle r^{2^k}\rangle_{pq}$, for some $j\in\{0,\ldots,2^{k-1}-1\}$, which implies that $r^{2^ks+j}=2$ (mod $pq$), for some $s\in\{0,\ldots,|\langle r^{2^k}\rangle_{pq}|-1\}$. Therefore
	\begin{center}
		$\Psi(2)=\Psi(r^{2^ks+j}_{pq})=(r^{2^ks+j}_p,r^{2^ks+j}_q)\neq(2,2)$,
	\end{center}
which is a contradiction. Hence $2\in r^{2^{k-1}}\langle r^{2^k} \rangle_{pq}$.
\end{proof}

The author in \cite{AvilaAKCE} proved the following:

\begin{teo}\cite{AvilaAKCE}\label{thm:New_Skolem_pq}
Let $p,q\equiv3$ (mod 8) be odd prime numbers such that $p<q$ and $(p-1)\not|(q-1)$. And let $r\in\mathbb{Z}_p^*$ be a primitive root of $\mathbb{Z}_p^*$ and $\mathbb{Z}_q^*$. If $pS_{q}=\left\{\{px,2px\}:x\in QR(q)\right\}$, $qS_{p}=\left\{\{qx,2qx\}:x\in QR(p)\right\}$ and
$S_{pq}=\left\{\{x,2x\}:x\in\langle r^2\rangle_{pq}\right\}\cup\left\{\{\lambda x,2\lambda x\}:x\in\langle r^2\rangle_{pq}\right\}$, where $\lambda\not\in\langle r^2\rangle_{pq}\cup2\langle r^2\rangle_{pq}$, then set $S=pS_{q}\cup qS_{p}\cup S_{pq}$ is a strong Skolem starter for $\mathbb{Z}_{pq}$.
\end{teo}

\begin{coro}\cite{AvilaAKCE}\label{coro:New_Skolem}
Let $p,q\equiv3$ (mod 8) be odd prime numbers such that $p<q$ and $(p-1)\not|(q-1)$. And let $r\in\mathbb{Z}_p^*$ be a primitive root of $\mathbb{Z}_p^*$ and $\mathbb{Z}_q^*$. If $pS_{q}=\left\{\{px,2^{-1}px\}:x\in QR(q)\right\}$, $qS_{p}=\left\{\{qx,2^{-1}qx\}:x\in QR(p)\right\}$ and
$S_{pq}=\left\{\{x,2^{-1}x\}:x\in\langle r^2\rangle_{pq}\right\}\cup\left\{\{\lambda x,2^{-1}\lambda x\}:x\in\langle r^2\rangle_{pq}\right\}$, where $\lambda\not\in\langle r^2\rangle_{pq}\cup2\langle r^2\rangle_{pq}$, then set $\hat{S}=pS_{q}\cup qS_{p}\cup S_{pq}$ is a strong Skolem starter for $\mathbb{Z}_{pq}$.
\end{coro}

It is not hard to verify that $\hat{S}=-S$, since  $-1,2\in r\langle r^{2}\rangle_{pq}$ (see \cite{AvilaAKCE}), where $r\in\mathbb{Z}_p^*$ is a primitive root of $\mathbb{Z}_p^*$ and $\mathbb{Z}_q^*$. Theorem \ref{thm:New_Skolem_pq} gives an alternative (simple) construction of cardioidal starter for $\mathbb{Z}_{pq}$, where $p$ and $q$ are odd prime numbers. This cardioidal starters can be obtained from \cite{Shalaby}.

We extend Theorem \ref{thm:New_Skolem_pq} and Corollary \ref{coro:New_Skolem} for some infinite odd primes of the form $p\equiv1$ (mod 8).

\begin{teo}\label{thm:New_main_pq}
Let $t_1$ and $t_2$ be two odd integers greater than 1 such that $t_1<t_2$, and let $k\geq3$ be a integer. Suppose that $p=2^kt_1+1$ and $q=2^kt_2+1$ are prime numbers such that $p<q$ and $(p-1)\not|(q-1)$, and $r\in\mathbb{Z}_p^*$ be a primitive root of $\mathbb{Z}_p^*$ and $\mathbb{Z}_q^*$. If $2\in r^{2^{k-1}}\langle r^{2^k}\rangle_{p}$ and $2\in r^{2^{k-1}}\langle r^{2^k}\rangle_{q}$, and
	\begin{eqnarray*}
		pT_{q}&=&\left\{\{px,2px\}:x\in r^{j}\langle r^{2^k}\rangle_{q}, j=0,\ldots,2^{k-1}-1\right\}\\
		qT_{p}&=&\left\{\{px,2px\}:x\in r^{j}\langle r^{2^k}\rangle_{p}, j=0,\ldots,2^{k-1}-1\right\}\\
		T_{pq}&=&\left\{\{x,2x\}:x\in r^{j}\langle r^{2^k}\rangle_{pq}, j=0,\ldots,2^{k-1}-1\right\}\\
		&\cup&\left\{\{\lambda x,2\lambda x\}:x\in r^{j}\langle r^{2^k}\rangle_{pq}, j=0,\ldots,2^{k-1}-1\right\}
	\end{eqnarray*} 
where $\lambda\in r^{2^{k-1}}\langle r^{2^k}\rangle_{pq}\setminus\{2\}$, then set $T=pT_{q}\cup qT_{p}\cup T_{pq}$ is a strong Skolem starter for $\mathbb{Z}_{pq}$.
\end{teo}

\begin{proof}
	Let $r\in\mathbb{Z}_p^*$ be primitive roots of $\mathbb{Z}_p^*$ and $\mathbb{Z}_q^*$. We know that $p\mathbb{Z}_q^*$, $q\mathbb{Z}_p^*$ and $\underline{G}_{pq}$ forms a partition of $\mathbb{Z}_{pq}^*$. Let $qT_p$, $pT_q$ and $T_{pq}$ as above. By Lemma \ref{lema:New_-1}, it is easy to see that 
	\begin{eqnarray*}
		\{\pm px: x\in r^{j}\langle r^{2^k}\rangle_{q}, j=0,\ldots,2^{k-1}-1\}&=&p\mathbb{Z}_{q}^*\\
		\{\pm qx: x\in r^{j}\langle r^{2^k}\rangle_{p}, j=0,\ldots,2^{k-1}-1\}&=&q\mathbb{Z}_{p}^*\\
		\{\pm x: x\in r^{j}\langle r^{2^k}\rangle_{pq}, j=0,\ldots,2^{k-1}-1\}&\cup&\\
		\{\pm \lambda x: x\in r^{j}\langle r^{2^k}\rangle_{pq}, j=0,\ldots,2^{k-1}-1\}&=&\underline{G}_{pq}
	\end{eqnarray*} 
	Hence, the set $T=pT_{q}\cup qT_{p}\cup T_{pq}$ is a starter. Moreover, the starter $T$ is a cardioidal starter for $\mathbb{Z}_{pq}$. Hence, by Lemma \ref{lemma:Skolem} the starter $S$ is strong and Skolem.  
\end{proof}

\begin{coro}
Let $t_1$ and $t_2$ be two odd integers greater than 1 such that $t_1<t_2$, and let $k\geq3$ be a integer. Suppose that $p=2^kt_1+1$ and $q=2^kt_2+1$ are prime numbers such that $p<q$ and $(p-1)\not|(q-1)$, and $r\in\mathbb{Z}_p^*$ be a primitive root $\mathbb{Z}_p^*$ and $\mathbb{Z}_q^*$. If $2\in r^{2^{k-1}}\langle r^{2^k}\rangle_{p}$ and $2\in r^{2^{k-1}}\langle r^{2^k}\rangle_{q}$, and
	\begin{eqnarray*}
		p\hat{T}_{q}&=&\left\{\{px,2^{-1}px\}:x\in r^{j}\langle r^{2^k}\rangle_{q}, j=0,\ldots,2^{k-1}-1\right\}\\
		q\hat{T}_{p}&=&\left\{\{px,2^{-1}px\}:x\in r^{j}\langle r^{2^k}\rangle_{p}, j=0,\ldots,2^{k-1}-1\right\}\\
		\hat{T}_{pq}&=&\left\{\{x,2^{-1}x\}:x\in r^{j}\langle r^{2^k}\rangle_{pq}, j=0,\ldots,2^{k-1}-1\right\}\\
		&\cup&\left\{\{\lambda x,2^{-1}\lambda x\}:x\in r^{j}\langle r^{2^k}\rangle_{pq}, j=0,\ldots,2^{k-1}-1\right\}
	\end{eqnarray*} 
where $\lambda\in r^{2^{k-1}}\langle r^{2^k}\rangle_{pq}\setminus\{2\}$, then set $\hat{T}=p\hat{T}_{q}\cup q\hat{T}_{p}\cup\hat{T}_{pq}$ is a strong Skolem starter for $\mathbb{Z}_{pq}$.
\end{coro}

It is not hard to verify that $\hat{T}=-T$, since  $-1,2\in r^{2^{k-1}}\langle r^{2^k}\rangle_{pq}$, where $r\in\mathbb{Z}_p^*$ is a primitive root of $\mathbb{Z}_p^*$ and $\mathbb{Z}_q^*$. Theorem \ref{thm:New_Skolem_pq} gives an alternative (simple) construction of cardioidal starter for $\mathbb{Z}_{pq}$, for infinitely many odd primes $p,q\equiv1$ (mod 8) and $n\geq1$ an integer number. This cardioidal starters can be obtained from \cite{Shalaby}.
%%%%%%%%%%%%%%%%%%%%%%%%%%%%%%%%%%%%%%%%%%%%%%%%%%%%
%%%%%%%%%%%%%%%%%%%%%%%%%%%%%%%%%%%%%%%%%%%%%%%%%%%
%%%%%%%%%%%%%%%%%%%%%%%%%%%%%%%%%%%%%%%%%%%%%%%%%%%%%
\subsection*{Conclusion}
In this note we present an alternative (simple) construction of cardioidal starters (strong and Skolem) for $\mathbb{Z}_{q^n}$, where $q\equiv3$ (mod 8) is an odd prime number and $n\geq1$ is an integer number. Also, we present an alternative construction of cardioidal starters (strong and Skolem) for $\mathbb{Z}_p$, $\mathbb{Z}_{pq}$ and $\mathbb{Z}_{p^n}$, for infinitely many odd primes $p,q\equiv1$ (mod 8) and $n\geq1$ an integer number. This cardioidal starters can be obtained by a more general result from \cite{Shalaby}. This work is inspire of the work given Ogandzhanyants et al. in \cite{Shalaby}.

{\bf Acknowledgment}

Research was partially supported by SNI, México and CONACyT, México.

\bibliographystyle{amsplain}

\end{document}